\documentclass[fleqn,final,1p,times,authoryear]{elsarticle}
\usepackage{amsmath}
\usepackage{amsthm}
\usepackage{amssymb}
\usepackage{amsfonts}

\usepackage{enumitem,color}
\usepackage{algorithm2e}
\usepackage{euscript}

\RequirePackage{color}
\usepackage[pdfstartview=FitH,
            colorlinks,
           linkcolor=reference,
            citecolor=citation,
            urlcolor=e-mail]{hyperref}
\definecolor{e-mail}{rgb}{0,.40,.80}
\definecolor{reference}{rgb}{.20,.60,.22}
\definecolor{citation}{rgb}{0,.40,.80}

\newtheorem{theorem}{Theorem}
\newtheorem{lemma}[theorem]{Lemma}
\newtheorem{corollary}[theorem]{Corollary}

\newtheorem{proposition}[theorem]{Proposition}
\newtheorem{definition}[theorem]{Definition}

\newtheorem{example}[theorem]{Example}
\newtheorem{remark}[theorem]{Remark}

\newcommand{\ZN}{{\mathbb Z}_{\geqslant 0}}
\newcommand{\Z}{{\mathbb Z}}

\newcommand{\K}{\textbf{k}}

\DeclareMathOperator{\id}{id}
\DeclareMathOperator{\ord}{ord}
\DeclareMathOperator{\rank}{rank}
\DeclareMathOperator{\antik}{anti-{\it k}}
\DeclareMathOperator{\lead}{lead}
\DeclareMathOperator{\init}{init}
\DeclareMathOperator{\sepp}{sep}

\DeclareMathOperator{\lcd}{lcd}

\begin{document}

\begin{frontmatter}

\title{New order bounds in differential elimination algorithms\tnoteref{t1}}
\tnotetext[t1]{This work was partially supported by the NSF grants CCF-095259,  CCF-1563942, DMS-1606334, by the NSA grant \#H98230-15-1-0245, by CUNY CIRG \#2248, by PSC-CUNY grant \#69827-00 47, by the Austrian Science Fund FWF grant Y464-N18.}
\author{Richard Gustavson}
\address{CUNY Graduate Center, Ph.D. Program in Mathematics, 365 Fifth Avenue, New York, NY 10016, USA}
\ead{rgustavson@gradcenter.cuny.edu}

\author{Alexey Ovchinnikov}
\address{CUNY Queens College, Department of Mathematics, 65-30 Kissena Boulevard, Queens, NY 11367, USA}
\ead{aovchinnikov@qc.cuny.edu}

\author{Gleb Pogudin}
\address{Johannes Kepler University, Institute for Algebra, Science Park II, 3rd Floor, 4040 Linz, Austria}
\ead{pogudin@algebra.uni-linz.ac.at}

\begin{abstract}
We present a new upper bound for the orders of derivatives in the Rosenfeld-Gr\"obner algorithm. This algorithm computes a regular decomposition of a radical differential ideal in the ring of differential polynomials over a differential field of characteristic zero with an arbitrary number of commuting derivations.  This decomposition can then be used to test for membership in the given radical differential ideal. In particular, this algorithm allows us to determine whether a system of polynomial PDEs is consistent.

Previously, the only known order upper bound was given by Golubitsky, Kondratieva, Moreno Maza, and Ovchinnikov for the case of a single derivation.  We achieve our bound by associating to the algorithm antichain sequences whose lengths can be bounded using the results of Le\'on S\'anchez and Ovchinnikov.
\end{abstract}

\begin{keyword}
Polynomial differential equations; differential elimination algorithms; computational complexity
\end{keyword}

\end{frontmatter}

%%%%%%%%%%%%%%%%%%%%%%%%%%%%%%%

\section{Introduction}
The Rosenfeld-Gr\"obner algorithm is a fundamental algorithm in the  algebraic theory of differential equations.  This algorithm, which first appeared in \citep{BLOP,Boulier2009}, takes as its input a finite set $F$ of differential polynomials and outputs a representation of the radical differential ideal generated by $F$ as a finite intersection of regular differential ideals.  The algorithm has many applications; for example, it can be used to test membership in a radical differential ideal, and, in conjunction with the differential Nullstellensatz, can be used to test the consistency of a system of polynomial differential equations.  See \citep{GKMO} for a history of the development of the Rosenfeld-Gr\"obner algorithm and similar decomposition algorithms.

The Rosenfeld-Gr\"obner algorithm has been implemented in  {\sc Maple} as a part of the {\tt DifferentialAlgebra} package.  In order to determine the complexity of the algorithm, we need to (among other things) find an upper bound on the orders of derivatives that appear in all intermediate steps and in the output of the algorithm.  The first step in answering this question was completed in \citep{GKMO}, in which an upper bound in the case of a single derivation and any ranking on the set of derivatives was found.  If there are $n$ unknown functions and the order of the original system is $h$, the authors showed that an upper bound on the orders of the output of the Rosenfeld-Gr\"obner algorithm is $h(n-1)!$.

In this paper, we extend this result by finding an upper bound for the orders of derivatives that appear in the intermediate steps and in the output of the Rosenfeld-Gr\"obner algorithm in the case of an arbitrary number of commuting derivations and a weighted ranking on the derivatives.  We first compute an upper bound for the weights of the derivatives involved for an arbitrary weighted ranking; by choosing a specific weight, we obtain an upper bound for the orders of the derivatives.  For this, we construct special antichain sequences in the set $\ZN^m \times \{1,\ldots,n\}$ equipped with a specific partial order. We then use \citep{LSO} to estimate the lengths of our sequences. A general analysis of lengths of antichain sequences began in \citep{Pierce} and continued in \citep{FLS}.

We show that an upper bound for the weights of derivatives in the intermediate steps and in the output of the Rosenfeld-Gr\"obner algorithm is given by $hf_{L+1}$, where $h$ is the weight of our input system of differential equations, $\{f_0,f_1,f_2,\dots\}$ is the Fibonacci sequence $\{0,1,1,2,3,5,\dots\}$, and $L$ is the maximal possible length of a certain antichain sequence (that depends solely on $h$, the number $m$ of derivations, and the number $n$ of unknown functions).  For $m=2$, we refine this upper bound in a new way by showing that the weights of the derivatives in question are bounded above by 
a sequence defined similarly to the Fibonacci sequence but with a slower growth rate.

By choosing a specific weight, we are able to produce an upper bound for the orders of the derivatives in the intermediate steps and in the output of the Rosenfeld-Gr\"obner algorithm.  Note that this bound is different from the upper bounds for the effective differential Nullstellensatz \citep{DJS,GKO}, which are higher and also depend on the degree of the given system of differential equations.  
Our result is an improvement of \citep{GOP} because it allows us to compute sharper order upper bounds with respect to specific derivations than the previous upper bound did, and because of the refinement in the  case $m=2$.
For example, if $n=2$ and $h=3,4,5$, the new bound is 3, 8, 33 times better, respectively.

The paper is organized as follows.  In Section~\ref{sec:background}, we present the background material from differential algebra that is necessary to understand the Rosenfeld-Gr\"obner algorithm.  In Section~\ref{sec:RGalg}, we describe this algorithm as it is presented in \citep{Hubert}, as well as two necessary auxiliary algorithms.  In Section~\ref{sec:bound}, we prove our main result on the upper bound.  In Section~\ref{sec:specific}, we calculate the upper bound for specific values using the results of \citep{LSO}.  In Section~\ref{sec:lower}, we give an example showing that the lower bound for the orders of derivatives in the Rosenfeld-Gr\"obner algorithm is at least double-exponential in the number of derivations.

%%%%%%%%%%%%%%%%%%%%%%%%%%%%%%%%%%%%%%%%%%%%%%%%%%%%%%%%%%%%%%%%%%%%%%%%%%%%%%%%%%%%%%%%%%%%%%%%%%%%%%%%%%%%%%%

\section{Background on differential algebra}\label{sec:background}
In this section, we present background material from  differential algebra that is pertinent to the Rosenfeld-Gr\"obner algorithm.  For a more in-depth discussion, we refer the reader to \citep{Hubert, Kolchin}.

\begin{definition}
A \emph{differential ring} is a commutative ring $R$ with a collection of $m$ commuting derivations $\Delta = \{\partial_1,\dots,\partial_m\}$ on $R$.  
\end{definition}
\begin{definition}
An ideal $I$ of a differential ring is a \emph{differential ideal} if  $\delta a \in I$ for all $a \in I$, $\delta \in \Delta$.
\end{definition}
For a set $A \subseteq R$, let $(A)$, $\sqrt{(A)}$, $[A]$, and $\{A\}$ denote the smallest ideal, radical ideal, differential ideal, and radical differential ideal containing $A$, respectively.  If $\mathbb{Q} \subseteq R$, then $\{A\} = \sqrt{[A]}$.

\begin{remark} In this paper, as usual, we also use the braces $\{a_1, a_2,\ldots\}$ to denote the set containing the elements $a_1,a_2,\ldots$. Even though this notation conflicts with the above notation for radical differential ideals (used here for historical reasons), it will be clear from the context which of the two objects we mean in each particular situation.
\end{remark}

In this paper, $\K$ is a differential field of characteristic zero with $m$ commuting derivations.  The set of derivative operators is denoted by
\[
\Theta := \left\{\partial_1^{i_1}\cdots \partial_m^{i_m}: i_j \in \ZN, 1 \leqslant j \leqslant m\right\}.
\]
For $Y = \{y_1,\dots,y_n\}$ a set of $n$ {\em differential indeterminates}, the set of derivatives of $Y$ is 
\[
\Theta Y := \{\theta y : \theta \in \Theta, y \in Y\}.
\]
Then the ring of \emph{differential polynomials} over $\K$ is defined to be
\[
\K\{Y\} = \K\{y_1,\dots,y_n\} := \K[\theta y : \theta y \in \Theta Y].
\]
We can naturally extend the derivations $\partial_1,\dots,\partial_m$ to the ring $\K\{Y\}$ by defining 
\[
\partial_j \left(\partial_1^{i_1}\cdots \partial_m^{i_m} y_k\right) := \partial_1^{i_1}\cdots \partial_j^{i_j+1}\cdots \partial_m^{i_m}y_k.
\]
For any $\theta = \partial_1^{i_1}\cdots \partial_m^{i_m} \in \Theta$, we define the \emph{order} of $\theta$ to be
\[
\ord(\theta) := i_1 + \cdots + i_m.
\]
For any derivative $u = \theta y \in \Theta Y$, we define \[\ord(u) := \ord(\theta).\]  For a differential polynomial $f \in \K\{Y\} \setminus \K$, we define the order of $f$ to be the maximum order of all derivatives that appear in $f$.  For any finite set $A \subseteq \K\{Y\} \setminus \K$, we set
\begin{equation}\label{eq:HA}
\mathcal{H}(A) := \max\{\ord(f) : f \in A\}.
\end{equation}
For any $\theta = \partial_1^{i_1}\cdots \partial_m^{i_m}$ and positive integers $c_1,\dots,c_m \in \Z_{>0}$, we define the \emph{weight} of $\theta$ to be 
\[
w(\theta) = w\left(\partial_1^{i_1}\cdots\partial_m^{i_m}\right) := c_1i_1 + \cdots + c_mi_m.
\]
Note that if all of the $c_i = 1$, then $w(\theta) = \ord(\theta)$ for all $\theta \in \Theta$.  For a derivative $u = \theta y \in \Theta Y$, we define the weight of $u$ to be $w(u) := w(\theta)$.  For any differential polynomial $f \in \K\{Y\} \setminus \K$, we define the weight of $f$, $w(f)$, to be the maximum weight of all derivatives that appear in $f$.  For any finite set $A \subseteq \K\{Y\} \setminus \K$, we set
\[
\mathcal{W}(A) := \max\{w(f) : f \in A\}.
\]

\begin{definition}
A \emph{ranking} on the set $\Theta Y$ is a total order $<$ satisfying the following two additional properties:  for all $u,v \in \Theta Y$ and all $\theta \in \Theta$, $\theta\ne\id$,
\[
u < \theta u\quad\text{and}\quad
u < v\implies \theta u < \theta v.
\]
A ranking $<$ is called an \emph{orderly ranking} if for all $u,v \in \Theta Y$,
\[\ord(u) < \ord(v)\implies u < v.
\]
Given a weight $w$, a ranking $<$ on $\Theta Y$ is called a \emph{weighted ranking} if for all $u,v \in \Theta Y$,
\[w(u) < w(v)\implies u < v.\]
\end{definition}

\begin{remark}Note that if $w\left(\partial_1^{i_1}\cdots\partial_m^{i_m}\right) = i_1 + \cdots + i_m$ (that is, $w(\theta) = \ord(\theta)$), then a weighted ranking $<$ on $\Theta Y$ is in fact an orderly ranking.
\end{remark}
From now on, we fix a weighted ranking $<$ on $\Theta Y$.
\begin{definition} Let $f \in \K\{Y\} \setminus \K$.
\begin{itemize}
\item
 The derivative $u \in \Theta Y$ of highest rank appearing in $f$ is called the \emph{leader} of $f$, denoted $\lead(f)$.  
\item
If we write $f$ as a univariate polynomial in $\lead(f)$, the leading coefficient is called the \emph{initial} of $f$, denoted $\init(f)$.  
\item
If we apply any derivative $\delta \in \Delta$ to $f$, the leader of $\delta f$ is $\delta (\lead(f))$, and the initial of $\delta f$ is called the \emph{separant} of $f$, denoted $\sepp(f)$.
\end{itemize}
\end{definition}

Given a set $A \subseteq \K\{Y\} \setminus \K$, we will denote the set of leaders of $A$ by $\mathfrak{L}(A)$, the set of initials of $A$ by $I_A$, and the set of separants of $A$ by $S_A$; we then let $H_A = I_A \cup S_A$ be the set of initials and separants of $A$.  

For a derivative $u \in \Theta Y$, we let $(\Theta Y)_{<u}$ (respectively, $(\Theta Y)_{\leqslant u}$) be the collection of all derivatives $v \in \Theta Y$ with $v < u$ (respectively, $v \leqslant u$).  For any derivative $u \in \Theta Y$, we let $A_{< u}$ (respectively, $A_{\leqslant u}$) be the elements of $A$ with leader $< u$ (respectively, $\leqslant u$), that is,
\[
A_{< u} := A \cap \K[(\Theta Y)_{< u}] \quad \mbox{and} \quad A_{\leqslant u} := A \cap \K[(\Theta Y)_{\leqslant u}].
\]
We can similarly define $(\Theta A)_{< u}$ and $(\Theta A)_{\leqslant u}$, where 
\[
\Theta A := \{\theta f : \theta \in \Theta, f \in A\}.
\]
Given $f \in \K\{Y\} \setminus \K$ such that $\deg_{\lead(f)}(f) = d$, we define the \emph{rank} of $f$ to be 
\[
\rank(f) := \lead(f)^d.
\]
The weighted ranking $<$ on $\Theta Y$ determines a pre-order (that is, a relation satisfying all of the properties of an order, except for the property that $a \leqslant b$ and $b \leqslant a$  imply that $a=b$) on $\K\{Y\} \setminus \K$:  

\begin{definition}
Given $f_1, f_2 \in \K\{Y\} \setminus \K$, we say that 
\[
\rank(f_1) < \rank(f_2)
\]
if $\lead(f_1) < \lead(f_2)$ or if $\lead(f_1) = \lead(f_2)$ and $\deg_{\lead(f_1)}(f_1) < \deg_{\lead(f_2)}(f_2)$.
\end{definition}

\begin{definition}
A differential polynomial $f$ is \emph{partially reduced} with respect to another differential polynomial $g$ if no proper derivative of $\lead(g)$ appears in $f$, and $f$ is \emph{reduced} with respect to $g$ if, in addition, 
\[
\deg_{\lead(g)}(f) < \deg_{\lead(g)}(g).
\]
A differential polynomial is then (partially) reduced with respect to a set $A \subseteq \K\{Y\} \setminus \K$ if it is (partially) reduced with respect to every element of $A$.
\end{definition}

\begin{definition}
For a set $A \subseteq \K\{Y\}\setminus \K$,  we say that $A$ is:
\begin{itemize}
\item \emph{autoreduced} if every element of $A$ is reduced with respect to every other element.
\item \emph{weak d-triangular} if $\mathfrak{L}(A)$ is autoreduced.
\item \emph{d-triangular} if $A$ is weak d-triangular and every element of $A$ is partially reduced with respect to every other element.
\end{itemize}
\end{definition}

Note that every autoreduced set is d-triangular.  Every weak d-triangular set (and thus every d-triangular and autoreduced set) is finite \citep[Proposition~3.9]{Hubert}.  Since the set of leaders of a weak d-triangular set $A$ is autoreduced, distinct elements of $A$ must have distinct leaders.  If $u \in \Theta Y$ is the leader of some element of a weak d-triangular set $A$, we let $A_u$ denote this element.

\begin{definition}
We define a pre-order on the collection of all weak d-triangular sets, which  we also call \emph{rank}, as follows.  Given two weak d-triangular sets $A=\{A_1,\dots,A_r\}$ and $B = \{B_1,\dots,B_s\}$, in each case arranged in increasing rank, we say that $\rank(A) < \rank(B)$ if either:
\begin{itemize}
\item there exists a $k \leqslant \min(r,s)$ such that $\rank(A_i) = \rank(B_i)$ for all $1 \leqslant i < k$ and $\rank(A_k) < \rank(B_k)$, or
\item $r > s$ and $\rank(A_i)= \rank(B_i)$ for all $1 \leqslant i \leqslant s$.
\end{itemize}
We also say that $\rank(A) = \rank(B)$ if $r=s$ and $\rank(A_i) = \rank(B_i)$ for all $1 \leqslant i \leqslant r$.
\end{definition}

We can restrict this ranking to the collection of all d-triangular sets or the collection of all autoreduced sets. 

\begin{definition}
A \emph{characteristic set} of a differential ideal $I$ is an autoreduced set $C \subseteq I$ of minimal rank among all autoreduced subsets of $I$.
\end{definition}

Given a finite set $S \subseteq \K\{Y\}$, let $S^\infty$ denote the multiplicative set containing $1$ and generated by $S$.  For an ideal $I \subseteq \K\{Y\}$, we define the \emph{colon ideal} to be
\[
I : S^\infty := \{a \in \K\{Y\} : \exists s \in S^\infty \mbox{ with } sa \in I\}.
\]
If $I$ is a differential ideal, then $I:S^\infty$ is also a differential ideal  \citep[Section~I.2]{Kolchin}.

\begin{definition}
For a differential polynomial $f \in \K\{Y\}$ and a weak d-triangular set $A \subseteq \K\{Y\}$, a \emph{differential partial remainder} $f_1$ and a \emph{differential remainder} $f_2$ of $f$ with respect to $A$ are differential polynomials such that there exist $s \in S_A^\infty$, $h \in H_A^\infty$ such that $sf \equiv f_1 \mod [A]$ and $hf \equiv f_2 \mod [A]$, with $f_1$ partially reduced with respect to $A$ and $f_2$ reduced with respect to $A$. 
\end{definition}

We denote a differential partial remainder of $f$ with respect to $A$ by $\mbox{pd-red}(f,A)$ and a differential remainder of $f$ with respect to $A$ by $\mbox{d-red}(f,A)$.  There are algorithms to compute $\mbox{pd-red}(f,A)$ and $\mbox{d-red}(f,A)$ for any $f$ and $A$  \citep[Algorithms~3.12 and~3.13]{Hubert}.  These algorithms have the property that 
\[
\rank(\mbox{pd-red}(f,A)),\  \rank(\mbox{d-red}(f,A)) \leqslant \rank(f);
\] 
since we have a weighted ranking, this implies that 
\[
w(\mbox{pd-red}(f,A)),\  w(\mbox{d-red}(f,A))\leqslant w(f).
\]

\begin{definition}
Two derivatives $u,v \in \Theta Y$ are said to have a \emph{common derivative} if there exist $\phi, \psi \in \Theta$ such that $\phi u = \psi v$.
Note this is the case precisely when $u = \theta_1 y$ and $v = \theta_2 y$ for some $y \in Y$ and $\theta_1,\theta_2 \in \Theta$. 
\end{definition}

\begin{definition}
If $u = \partial_1^{i_1}\cdots \partial_m^{i_m}y$ and $v = \partial_1^{j_1}\cdots \partial_m^{j_m}y$ for some $y \in Y$, we define the \emph{least common derivative} of $u$ and $v$, denoted $\lcd(u,v)$, to be
\[
\lcd(u,v) = \partial_1^{\max(i_1,j_1)}\cdots \partial_m^{\max(i_m,j_m)}y.
\]
\end{definition}

\begin{definition}
For $f, g \in \K\{Y\}\setminus \K$,  we define the \emph{$\Delta$-polynomial} of $f$ and $g$, denoted $\Delta(f,g)$, as follows.  If $\lead(f)$ and $\lead(g)$ have no common derivatives, set $\Delta(f,g) = 0$.  Otherwise, let $\phi, \psi \in \Theta$ be such that \[\lcd(\lead(f),\lead(g)) = \phi(\lead(f)) = \psi(\lead(g)),\] and define
\[
\Delta(f,g) := \sepp(g)\phi(f) - \sepp(f)\psi(g).
\]
\end{definition}

\begin{definition}
A pair $(A,H)$ is called a \emph{regular differential system} if:
\begin{itemize}
\item $A$ is a d-triangular set
\item $H$ is a set of differential polynomials that are all partially reduced with respect to $A$
\item $S_A \subseteq H^\infty$
\item for all $f,g \in A$, 
$
\Delta(f,g) \in ((\Theta A)_{< u}):H^\infty
$,
where $u = \lcd(\lead(f),\lead(g))$.
\end{itemize}
\end{definition}

\begin{definition}
Any ideal of the form $[A]:H^\infty$, where $(A,H)$ is a regular differential system, is called a \emph{regular differential ideal}.
\end{definition}
Every regular differential ideal is a radical differential ideal \citep[Theorem~4.12]{Hubert}. 

\begin{definition}
Given a radical differential ideal $I \subseteq \K\{Y\}$, a \emph{regular decomposition} of $I$ is a finite collection of regular differential systems $\{(A_1,H_1),\dots,(A_r,H_r)\}$ such that 
\[
I = \bigcap_{i = 1}^r [A_i]:H_i^\infty.
\]
\end{definition}
Due to the Rosenfeld-Gr\"obner algorithm, every radical differential ideal in $\K\{Y\}$ has a regular decomposition.

\begin{definition}
A d-triangular set $C$ is called a \emph{differential regular chain} if it is a characteristic set of $[C]:H_C^\infty$; in this case, we call $[C]:H_C^\infty$ a \emph{characterizable differential ideal}.  
\end{definition}

\begin{definition}
A \emph{characteristic decomposition} of a radical differential ideal $I \subseteq \K\{Y\}$ is a representation of $I$ as an intersection of characterizable differential ideals.
\end{definition}

As we will recall in Section~\ref{sec:RGalg}, every radical differential ideal also has a characteristic decomposition.

%%%%%%%%%%%%%%%%%%%%%%%%%%%%%%%%%%%%%%%%%%%%%%%%%%%%%%%%%%%%%%%%%%%%%%%%%%%%%%%%%%%%%%%%%%%%%%%%%%%%%%%%%%%%%%%%%%%%%%%%%%%%%%%%%%%%

\section{Rosenfeld-Gr\"obner algorithm}\label{sec:RGalg}

Below we reproduce the {\sf Rosenfeld-Gr\"obner} algorithm from \citep[Section~6]{Hubert}.  This algorithm relies on two others, called {\sf auto-partial-reduce} and {\sf update}, which we also include.  We include these two auxiliary algorithms because, in Section~\ref{sec:bound}, we will study their effect on the growth of the weights of derivatives in  {\sf Rosenfeld-Gr\"obner}.

{\sf Rosenfeld-Gr\"obner} takes as its input two finite subsets $F,K \in \K\{Y\}$ and outputs a finite set $\mathcal{A}$ of regular differential systems such that
\begin{equation}\label{regulardecomp}
\{F\}:K^\infty = \bigcap_{(A,H) \in \mathcal{A}} [A]:H^\infty,
\end{equation}
where $\mathcal{A} = \varnothing$ if $1 \in \{F\}:K^\infty$.

If we have a decomposition of $\{F\}:K^\infty$ as in \eqref{regulardecomp}, we can compute, using only algebraic operations, a decomposition of the form
\begin{equation}\label{chardecomp}
\{F\}:K^\infty = \bigcap_{C \in \mathcal{C}}[C]:H_C^\infty,
\end{equation}
where $\mathcal{C}$ is finite and each $C \in \mathcal{C}$ is a differential regular chain \citep[Algorithms~7.1 and~7.2]{Hubert}.  This means that an upper bound on $\bigcup_{(A,H) \in \mathcal{A}} \mathcal{W}(A \cup H)$ from \eqref{regulardecomp} will also be an upper bound on $\bigcup_{C \in \mathcal{C}} \mathcal{W}(C)$ from \eqref{chardecomp}.

{\sf Rosenfeld-Gr\"obner} has many immediate applications.  For example, if $K = \{1\}$, then $\{F\}:K^\infty = \{F\}$, so in this case, {\sf  Rosenfeld-Gr\"obner} computes a regular decomposition of $\{F\}$, which then also gives us a characteristic decomposition of $\{F\}$ by the discussion in the previous paragraph.

The \emph{weak differential Nullstellensatz} says that a system of polynomial differential equations $F = 0$ is consistent (that is, has a solution in some differential field extension of $\K$) if and only if $1 \notin [F]$  \citep[Section~IV.2]{Kolchin}.  Thus, since $\mbox{\sf Rosenfeld-Gr\"obner}(F,K) = \varnothing$ if and only if $1 \in \{F\}:K^\infty$, we see that $F=0$ is consistent if and only if $\mbox{\sf Rosenfeld-Gr\"obner}(F,\{1\}) \neq \varnothing$.

More generally, {\sf Rosenfeld-Gr\"obner} and its extension for computing a characteristic decomposition of a radical differential ideal allow us to test for membership in a radical differential ideal, as follows.  Suppose we have computed a characteristic decomposition
\[
\{F\} = \bigcap_{C \in \mathcal{C}} [C]:H_C^\infty.
\]
Now, a differential polynomial $f \in \K\{Y\}$ is contained in $\{F\}$ if and only if $f \in [C]:H_C^\infty$ for all $C \in \mathcal{C}$; this latter case is true if and only if $\mbox{d-red}(f,C) = 0$, which can be tested using \citep[Algorithm~3.13]{Hubert}.

{\sf Rosenfeld-Gr\"obner}, {\sf auto-partial-reduce}, and {\sf update} rely on the following tuples of differential polynomials:
\begin{definition}
A \emph{Rosenfeld-Gr\"obner quadruple} (or \emph{RG-quadruple}) is a 4-tuple $(G,D,A,H)$ of finite subsets of $\K\{Y\}$ such that:
\begin{itemize}
\item $A$ is a weak d-triangular set, $H_A \subseteq H$, $D$ is a set of $\Delta$-polynomials, and
\item for all $f,g \in A$, either $\Delta(f,g) = 0$ or $\Delta(f,g) \in D$ or 
\[
\Delta(f,g) \in \left(\Theta(A \cup G)_{<u}\right):H_u^\infty,\]
where $u = \lcd(\lead(f),\lead(g))$ and 
$
H_u = H_{A_{<u}} \cup (H \setminus H_A) \cap \K[(\Theta Y)_{<u}]
$.
\end{itemize}
\end{definition}

\begin{algorithm}
	\TitleOfAlgo{{\sf Rosenfeld-Gr\"obner}, \citep[Algorithm~6.11]{Hubert}}
    \label{alg:rosenfeld_groebner}
	\KwData{ $F$, $K$ finite subsets of $\K \{Y\}$ }
	\KwResult{
		A set $\mathcal{A}$ of regular differential systems such that:
        \begin{itemize}
			\item $\mathcal{A}$ is empty if it has been detected that $1 \in \{ F \}\colon K^{\infty}$
	    	\item $\{ F \} \colon K^{\infty} = \bigcap\limits_{(A, H) \in \mathcal{A}} [A] : H^{\infty}$ otherwise
        \end{itemize}
	}
    $\mathcal{S} := \{ (F, \varnothing, \varnothing, K) \} $\;
    $\mathcal{A} := \varnothing$\;
    \While{ $\mathcal{S} \neq \varnothing$ }{
    	$(G, D, A, H) := \mbox{an element of } \mathcal{S}$\;
        $\bar{\mathcal{S}} = \mathcal{S} \setminus (G, D, A, H)$\;
        \eIf{ $G\cup D = \varnothing$ }{
        	$\mathcal{A} := \mathcal{A} \cup \mbox{\sf auto-partial-reduce}(A, H)$\;
        }{
        	$p := \mbox{an element of } G\cup D$\;
            $\bar{G}, \bar{D} := G\setminus \{ p \}, D \setminus \{ p \}$\;
            $\bar{p} := \mbox{d-red}(p, A)$\;
            \eIf{ $\bar{p} = 0$ }{
            	$\bar{\mathcal{S}} := \bar{\mathcal{S}} \cup \{ (\bar{G}, \bar{D}, A, H) \}$\;
            }{\If{ $\bar{p} \notin \K$}{
            	$\bar{p}_i := \bar{p} - \init(\bar{p})\rank(\bar{p})$
                $\bar{p}_s := \deg_{\lead(\bar{p})}(\bar{p})\bar{p} - \lead(\bar{p})\sepp(\bar{p})$\;
                $\bar{\mathcal{S}} := \bar{\mathcal{S}} \cup \{ \text{\sf update}(\bar{G}, \bar{D}, A, H, \bar{p}), (G \cup \{ \bar{p}_s, \sepp(\bar{p}) \}, \bar{D}, A, H \cup \{ \init(\bar{p})\}), (\bar{G} \cup \{ \bar{p}_i, \init(\bar{p}) \}, \bar{D}, A, H) \}$\;
                }
            }
        }
        $\mathcal{S} := \bar{\mathcal{S}}$\;
    }
    \KwRet{$\mathcal{A}$}\;
\end{algorithm}

\begin{algorithm}
\vspace*{-.1in}
	\TitleOfAlgo{{\sf auto-partial-reduce}, \citep[Algorithm~6.8]{Hubert}}
    \label{alg:auto-pd}
    \KwData{Two finite subsets $A,H$ of $\K\{Y\}$ such that $(\varnothing, \varnothing, A,H)$ is an RG-quadruple}
    \KwResult{
    \begin{itemize}
    	\item The empty set if it is detected that $1 \in [A]:H^\infty$
        \item Otherwise, a set with a single regular differential system $(B,K)$ with $\mathfrak{L}(A) = \mathfrak{L}(B)$, $H_B \subseteq K$, and $[A]:H^\infty = [B]:K^\infty$
    \end{itemize}
    }
    $B:=\varnothing$\;
    \For{ $u \in \mathfrak{L}(A)$ increasingly }{
    	$b:= \mbox{pd-red}(A_u,B)$\;
        \eIf{ $\rank(b) = \rank(A_u)$ }{
        	$B:=B\cup\{b\}$\;
        }{
        	\KwRet{$(\varnothing)$}\;
        }
    }
    $K:=H_B \cup \{\mbox{pd-red}(p,B) : p \in H \setminus H_A\}$\;
    \eIf{ $0 \in K$ }{
    	\KwRet{$(\varnothing)$}\;
    }{
    	\KwRet{$\{(B,K)\}$}\;
    }\vspace*{-0.25in}
\end{algorithm}

\begin{algorithm}
\vspace*{-.1in}
	\TitleOfAlgo{{\sf update} \citep[Algorithm~6.10]{Hubert}}
    \label{alg:update}
	\KwData{
    \begin{itemize}
    	\item A $4$-tuple $(G, D, A, H)$ of finite subsets of $\K\{Y\}$
        \item A differential polynomial $p$ reduced with respect to $A$ such that $(G \cup \{p\},D,A,H)$ is an RG-quadruple
    \end{itemize}
    }
    \KwResult{A new RG-quadruple $(\bar{G}, \bar{D}, \bar{A}, \bar{H})$}
    $u := \mbox{lead}(p)$\;
    $G_A := \{ a \in A \mid \lead(a) \in \Theta u\}$\;
    $\bar{A} := A \setminus G_A$\;
    $\bar{G} := G \cup G_A$\;
    $\bar{D} := D \cup \{ \Delta(p, a) \mid a \in \bar{A}\} \setminus \{ 0\}$\;
    $\bar{H} := H \cup \{ \sepp(p), \init(p)\}$\;
    \KwRet{ $(\bar{G}, \bar{D}, \bar{A} \cup \{p\}, \bar{H})$ }\;
\end{algorithm}

\begin{remark}
The RG-quadruple that is output by {\sf update} satisfies additional properties that we do not list, as they are not important for our analysis.  For more information, we refer the reader to \citep[Algorithm~6.10]{Hubert}
\end{remark}

%%%%%%%%%%%%%%%%%%%%%%%%%%%%%%%%%%%%%%%%%%%%%%%%%%%%%%%%%%%%%%%%%%%%%%%%%%%%%%%%%%%%%%%%%%%%%%%%%%%%%%%%%%%%%%%%%%%%%%%%%%%%%%%%%%%%%%

\section{Order upper bound}\label{sec:bound}

Given finite subsets $F,K \subseteq \K\{Y\}$, let $h = \mathcal{W}(F \cup K)$.  Our goal is to find an upper bound for
\[
\mathcal{W}\left(\bigcup_{(A,H) \in \mathcal{A}} (A \cup H)\right),
\]
where $\mathcal{A} = \mbox{\sf Rosenfeld-Gr\"obner}(F,K)$, in terms of $h$, $m$ (the number of derivations), and $n$ (the number of differential indeterminates).  By then choosing a specific weight, we can find an upper bound for $\mathcal{H}\left(\bigcup_{(A,H) \in \mathcal{A}}(A \cup H)\right)
$
in terms of $m$, $n$, and $\mathcal{H}(F \cup K)$.

We approach this problem as follows.  Every $(A,H) \in \mathcal{A}$ is formed by applying  {\sf auto-partial-reduce} to a $4$-tuple $(\varnothing, \varnothing, A',H') \in \mathcal{S}$.  Thus, it suffices:
\begin{itemize}
\item to bound how $\mbox{\sf auto-partial-reduce}$ increases the weight of a collection of differential polynomials (it turns out to not increase the weight), and
\item to bound $\mathcal{W}(G \cup D \cup A \cup H)$ for all $(G,D,A,H)$ added to $\mathcal{S}$ throughout the course of {\sf Rosenfeld-Gr\"obner}.
\end{itemize}
We accomplish the latter by determining when the weight of a tuple $(G,D,A,H)$ added to $\mathcal{S}$ is larger than the weights of the previous elements of $\mathcal{S}$ and bounding $\mathcal{W}(G \cup D \cup A \cup H)$ in this instance, and then bounding the number of times we can add such elements to $\mathcal{S}$. 

There is a sequence $\{ (G_i, D_i, A_i, H_i) \}_{i = 0}^N$ corresponding to each regular differential system $(A,H)$ in the output of {\sf Rosenfeld-Gr\"obner}, where $N = N_{(A,H)}$, such that $(G_{i + 1}, D_{i + 1}, A_{i + 1}, H_{i + 1})$ is obtained from $(G_i, D_i, A_i, H_i)$ during the while loop, $(G_0, D_0, A_0, H_0) = (F, \varnothing, \varnothing, K)$, and $(A,H) = \mbox{\sf auto-partial-reduce}(A_N, H_N)$.

We begin with an auxiliary result, which is an analogue of the \emph{first} property from \citep[Section~5.1]{GKOS}.
\begin{lemma}
\label{lem:reduced}
	For every $f \in A_i$ and $i < j$, there exists $g \in A_j$ such that $\lead(f) \in \Theta \lead(g)$.
       In particular, if $p$ is reduced with respect to $A_j$, then $p$ is  reduced with respect to $A_i$ for all $i < j$.
\end{lemma}

\begin{proof}
	It is sufficient to consider the case $j = i + 1$.
    If $(G_{i + 1}, D_{i + 1}, A_{i + 1}, H_{i + 1})$ was obtained from $(G_{i}, D_{i}, A_{i}, H_{i})$ without applying {\sf update}, then $A_i = A_{i + 1}$.
    Otherwise, either $f \in A_i \setminus G_{A_i}$ (we use the notation from {\sf update}), or $f \in G_{A_i}$.
    In the former case, $f \in A_{i + 1}$ as well, so we can set $g = f$.
    In the latter case, $\lead(f) \in \Theta\lead(p)$, so we can set $g = p$.
\end{proof}

We define a partial order $\preccurlyeq$ on the set of derivatives $\Theta Y$ as follows.  For $u, v \in \Theta Y$, we say that $u \preccurlyeq v$ if there exists $\theta \in \Theta$ such that $\theta u = v$.  Note that this implies that $u$ and $v$ are both derivatives of the same $y \in Y$.  

\begin{definition}
An \emph{antichain sequence} in $\Theta Y$ is a sequence of elements $S = \{s_1,s_2,\dots\} \subseteq \Theta Y$ that are pairwise incomparable in this partial order.   
\end{definition}

Given a sequence $\{ (G_i, D_i, A_i, H_i) \}_{i = 0}^N$ as above (where $N = N_{(A,H)}$ for some regular differential system $(A,H)$ in the output of {\sf  Rosenfeld-Gr\"obner}), we will construct an antichain sequence $S = \{ s_1, s_2, \ldots \} \subseteq \Theta Y$ inductively going along the sequence $\{ (G_i, D_i, A_i, H_i) \}$.
Suppose $S_{j-1} = \{s_1,\dots,s_{j-1}\}$ has been constructed after considering $(G_0, D_0, A_0, H_0), \ldots, (G_{i-1}, D_{i-1}, A_{i-1}, H_{i-1})$, where $S_0 = \varnothing$.  A $4$-tuple $(G_i, D_i, A_i, H_i)$ can be obtained from the tuple $(G_{i - 1}, D_{i - 1}, A_{i - 1}, H_{i - 1})$ in two ways:
\begin{enumerate}

    \item We did not perform {\sf update}.
    In this case, we do not append a new element to $S$.
    
    \item We performed {\sf update} with respect to a differential polynomial $\bar{p}$.
    If there exists $s_k \in S_{j-1}$ such that $\lead(\bar{p}) \leqslant s_k$, we do not append a new element to $S_{j-1}$.
    Otherwise, let $s_j = \lead(\bar{p})$ and define $S_j = \{s_1,\dots,s_j\}$.
    In the latter case, we set $k_j = i$. We also set $k_0 = 0$.
\end{enumerate}

\begin{theorem}\label{antichainbound}
    The sequence $\{s_j\}$ is an antichain sequence in $\Theta Y$ and, for all $j \geqslant 1$,
  \[
      w(s_j) \leqslant h f_{j},\] where $\{ f_j \}$ is the Fibonacci sequence.
\end{theorem}

	For $m = 2$, we provide a refined version of Theorem~\ref{antichainbound}.
    Let $\{f(n, h)_k\}$ be the sequence:
    \begin{equation}\label{eq:m2_bound}
    \begin{cases} f(n,h)_0=0,\ f(n, h)_1 = f(n, h)_2 = h \\ f(n, h)_{k} = f(n, h)_{k - 1} + f(n, h)_{k - 2} & \text{for }k \leqslant n + 1 \\ f(n, h)_{k} = f(n, h)_{k - 1} + f(n, h)_{k - 2} - 1 & \text{for } k > n + 1. \end{cases}
    \end{equation}
    
    \begin{proposition}\label{prop:m2_refined}
    	For $m = 2$ the sequence $\{ s_j\}$ satisfies, for all $j\geqslant 1$,
        \[
        w(s_j) \leqslant f(n, h)_j.
        \]
    \end{proposition}
    
    We will prove Proposition~\ref{prop:m2_refined} while proving Theorem~\ref{antichainbound}, highlighting the case $m = 2$.

\begin{proof}
	Let $i < j$.
    Assume that $s_j \succcurlyeq s_i$.
	Then, $p$ is not reduced with respect to $A_{k_i}$, which contradicts Lemma~\ref{lem:reduced}.
    On the other hand, the case $s_j \preccurlyeq s_i$ is impossible by the construction of the sequence, so $\{s_j\}$ is an antichain sequence.

	Let $L$  denote the length of the sequence $\{ s_j\}$.
	We denote  the maximal $j \in \ZN$ such that $k_j \leqslant i$ by $\antik_i$.
    For all $i \geqslant 0$, let us set $j = \antik_i$ and prove by induction on $i$ that
    \begin{enumerate}

    \item $\mathcal{W}\left( \bigcup\limits_{t = 0}^i (G_t \cup D_t \cup H_t) \right) \leqslant h f_{j + 1}$;
    \item $\mathcal{W}\left( \bigcup\limits_{t = 0}^i A_t \right) \leqslant h f_{j}$;
    \item For all distinct elements of $\bigcup\limits_{t = 0}^i A_t$, the weights of the least common derivatives of their leaders do not exceed $h f_{j + 1}$.
    \end{enumerate}
    
    	If $m = 2$, let $F_0 = 0$, $F_1 = F_2 = h$. We will show that there exists a sequence $\{F_r\}$ such that
        \begin{itemize}
        \item for all $r \geqslant 1$, $w(s_r)\leqslant F_r$ and
        \item $F_{r} = F_{r - 1} + F_{r - 2} - 1$ for all $r\geqslant 3$ except at most $n - 1$ of them, for which $F_{r} = F_{r - 1} + F_{r - 2}$. 
        In the latter case, we will say that $r$ is \emph{a jump index}.
        Note that~$2$ is not a jump index by the definition, although $F_2 = F_1 + F_0$.
        \end{itemize} For each such sequence, the induction hypothesis will be the following:
        \begin{enumerate}
		    \item\label{l:1} $\mathcal{W}\left( \bigcup\limits_{t = 0}^i (G_t \cup D_t \cup H_t) \right) \leqslant F_{j + 1}$ for $j < L$ and $\mathcal{W}\left( \bigcup\limits_{t = 0}^i (G_t \cup D_t \cup H_t) \right) \leqslant F_{L + 1} + 1$ for $j = L$;
		    \item $\mathcal{W}\left( \bigcup\limits_{t = 0}^i A_t \right) \leqslant F_{j}$;
		    \item\label{l:3} For all distinct elements of $\bigcup\limits_{t = 0}^i A_t$, the weights of the least common derivatives of their leaders do not exceed $F_{j + 1}$ for $j < L$ and $F_{L + 1} + 1$ for $j = L$;
            \item If, in either of~\eqref{l:1} or~\eqref{l:3}, the equality holds in the case $j = L$, then, for every $q$, $1 \leqslant q \leqslant n$, the sequence $\{ s_r\}$ contains $\partial_1^{a_q}y_q$ and $\partial_2^{b_q}y_q$ for some $a_q$ and $b_q$.
	    \end{enumerate}
    
    In the base case $i = 0 = k_0$, we have \[\mathcal{W}(G_0 \cup D_0 \cup H_0) = h = h f_1 \; (F_1\text{ in the case m = 2})\] and \[\mathcal{W}(A_0) = \mathcal{W}(\varnothing) = 0 = h f_0 \; (F_0\text{ in the case m = 2}).\]
    There are two distinct cases for $i + 1$:
    \begin{enumerate}
		\item \underline{Case $i + 1 < k_{j + 1}$} (so $\antik_{i+1} = j$). Then, $(G_{i + 1}, D_{i + 1}, A_{i + 1}, H_{i + 1})$ was obtained from $(G_{i}, D_{i}, A_{i}, H_{i})$ in one of the following ways:
        \begin{enumerate}
        	\item We did not perform {\sf update}.
            In this case, $A_{i + 1} = A_i$ and \[\mathcal{W}(G_{i + 1} \cup D_{i + 1} \cup H_{i + 1}) \leqslant \mathcal{W}(G_{i} \cup D_{i} \cup H_{i}).\]
            
            \item We performed {\sf update} with respect to a differential polynomial $p$ such that $\lead(f) \in \Theta\lead(p)$ for some $f \in \bigcup_{t = 0}^i A_t$.
            In this case, \[\mathcal{W}(A_{i + 1}) \leqslant \mathcal{W}\left( \bigcup_{t = 0}^i A_t \right).\]
            Then, for all $g \in A_t$ ($t \leqslant i$),
            \[w(\Delta(p,g))\leqslant w(\lcd(\lead(g),\lead(f))),\] which is bounded by $h f_{j + 1}$ (by $F_{j + 1}$ or $F_{L + 1} + 1$ in the case $m = 2$) due to the third inductive hypothesis.
			Since $D_{i + 1} \setminus D_i$ consists of some of these polynomials, $G_{i + 1} \setminus G_i \subseteq A_{i}$, and $H_{i + 1} \setminus H_i = \{ \sepp(p), \init(p) \}$, then \[\mathcal{W}(G_{i + 1} \cup D_{i + 1} \cup H_{i + 1}) \leqslant \mathcal{W}(G_i \cup D_i \cup H_i).\]
        \end{enumerate}
        
        \item \underline{Case $i + 1 = k_{j + 1}$} (so now $\antik_{i + 1} = j + 1$).
        We performed {\sf update} with respect to a differential polynomial $p$, which is a result of reduction of some $\tilde{p} \in G_{i} \cup D_i$ with respect to $A_i$.
        Then \[\mathcal{W}(A_{i + 1}) \leqslant \max(\mathcal{W}(A_{i}), w(p)) \leqslant h f_{j + 1}.\]
        Moreover, for every $g \in \bigcup\limits_{t = 0}^i A_t$,
        \begin{equation}\label{eq:lcd}
        w(\lcd(\lead(g), \lead(p)))\leqslant hf_{j} + hf_{j + 1} = h f_{j + 2}.
        \end{equation}
        Since $D_{i + 1} \setminus D_i$ consists of some of these polynomials, $G_{i + 1} \setminus G_i \subseteq A_{i}$, and $H_{i + 1} \setminus H_i = \{ \sepp(p), \init(p) \}$, we have 
        \begin{align*}\mathcal{W}(G_{i + 1} \cup& D_{i + 1} \cup H_{i + 1})\leqslant \max( \mathcal{W}(G_i \cup D_i \cup H_i), h f_{j + 2} ) = h f_{j + 2}.\end{align*}
        
        	In the case $m = 2$, instead of~\eqref{eq:lcd}, we obtain
            \begin{equation}\label{eq:lcd_m2}
            w(\lcd(\lead(g), \lead (p))) \leqslant w(\lead (p)) + w(\lead(g))
            \end{equation}
            If~\eqref{eq:lcd_m2} is strict, we have 
            \[
            w(\lcd(\lead(g), \lead (p))) \leqslant w(\lead(p)) + w(\lead(g)) - 1 \leqslant F_j + F_{j + 1} - 1 = F_{j + 2},
            \] 
            and $j + 2$ is not a jump index.
            Otherwise, \eqref{eq:lcd_m2} turns out to be an equality.
            In this case, the only possibility is $\lead(p) = \partial_1^a y_r$ and $\lead(g) = \partial_2^b y_r$ (or vice versa) for some $r$.
            Note that, for every $r$, such a situation occurs at most once.
            Consider the following two cases:
            \begin{enumerate}
	            \item For every $q$, $1 \leqslant q \leqslant n$, the sequence $s_1, \ldots, s_{j + 1}$ already contains $\partial_1^{a_q}y_q$ and $\partial_{2}^{b_q}y_q$ for some $a_q$ and $b_q$.
                In this case, $s_1, \ldots, s_{j + 1}$ already form an antichain sequence that cannot be extended further, so $j + 1 = L$.
                We set $F_{L + 1} = F_L + F_{L - 1} - 1$, so we can bound the right-hand side of~\eqref{eq:lcd_m2} from above by $F_{L + 1} + 1$.
                
                \item Otherwise, we just set $F_{j + 2} = F_{j + 1} + F_{j}$, so $j + 2$ is a jump index, and we still have less than $n$ of them.
            \end{enumerate}
        
	\end{enumerate}
    Since $w(s_j) \leqslant \mathcal{W}(A_{k_j}) \leqslant h f_{j}$, this completes the proof of Theorem~\ref{antichainbound}.

    	In order to complete the proof of Proposition~\ref{prop:m2_refined}, it is sufficient to show that, for every such sequence $\{F_j\}$, for all $j$, $f(n, h)_j \geqslant F_j$.
        Let $i_1, \ldots, i_{n - 1}$  denote the jump indices of $\{F_j\}$.
        Note that $\{f(n, h)_j\}$ is uniquely defined as a sequence of the same type as $\{F_j\}$ with jump indices $3, \ldots, n + 1$.
        It is sufficient to prove that, after decreasing any jump index of $\{F_j\}$ by one, we obtain a sequence which is not smaller than $\{ F_j\}$.
       Then, since we will obtain $\{f(n,h)_j\}$ after some number of such operations and the jump indices of $\{f(n,h)_j\}$ cannot be further decreased, we will have that $\{f(n,h)_j\}$ is the largest such sequence.
        The claim is true since, before decreasing $i_j$, the sequence was of the form 
        \[\ldots , F_{i_j - 2},\quad F_{i_j - 1} = F_{i_j - 3} + F_{i_j - 2} - 1,\quad F_{i_j} = F_{i_j - 1} + F_{i_j - 2} = F_{i_j - 3} + 2F_{i_j - 2} - 1, \ldots\]
        but, after decreasing $i_j$ by one, it will be of the form
        \[
        \ldots, F_{i_j - 2},\quad F_{i_j - 1} = F_{i_j - 3} + F_{i_j - 2},\quad F_{i_j} = F_{i_j - 1} + F_{i_j - 2} - 1 = F_{i_j - 3} + 2F_{i_j - 2} - 1, \ldots
        \]
        Since the rest of terms obey the same recurrence for both sequences, the latter is not smaller than the former.
\end{proof}
Let $\mathfrak{n}=\{ 1, \ldots, n\}$.  
Define the \emph{degree} of an element $((i_1,\dots,i_m),k) \in \ZN^m \times \mathfrak{n}$ to be $i_1 + \cdots + i_m$.  Given a weight $w\left(\partial_1^{i_1}\cdots\partial_m^{i_m}\right) = c_1i_1 + \cdots + c_mi_m$ on $\Theta$, define a map from the set of derivatives $\Theta Y$ to the set $\ZN^m \times \mathfrak{n}$ by
\[
\partial_1^{i_1}\cdots \partial_m^{i_m}y_k \mapsto ((c_1i_1,\dots,c_mi_m),k).
\]
Note the degree of the image of $\theta y$ in $\ZN \times \mathfrak{n}$ is equal to the weight of $\theta y$ in $\Theta Y$.

Under this map, the partial order $\preccurlyeq$ on $\Theta Y$ determines a partial order $\preccurlyeq$ on $\ZN^m \times \mathfrak{n}$ by saying \[((i_1,\dots,i_m),k) \preccurlyeq ((j_1,\dots,j_m),l) \iff k = l\ \text{and}\ i_r \leqslant j_r \ \text{for all}\ r,\ 1 \leqslant r \leqslant m.\]  Thus, every antichain sequence of $\Theta Y$ determines an antichain sequence of $\ZN^m\times \mathfrak{n}$.  Every antichain sequence of $\ZN^m \times \mathfrak{n}$ (and thus of $\Theta Y$) is finite \citep[Lemma~4.4]{Pierce}.

Given an increasing function $f:\Z_{>0} \to \ZN$, we say that $f$ \emph{bounds the degree growth} of an antichain sequence $S = \{s_1,\dots,s_k\}\subseteq \ZN^m \times \mathfrak{n}$ if $\deg(s_i) \leqslant f(i)$ for all $1 \leqslant i \leqslant k$.  By \citep[Lemma~4.9]{Pierce}, there is an upper bound on the length of an antichain sequence of $\ZN^m \times \mathfrak{n}$ with degree growth bounded by $f$, and this bound depends only on $m$, $n$, and $f$.  Let $\mathfrak{L}^n_{f,m}$ be the maximal length of an antichain sequence of $\ZN^m \times \mathfrak{n}$ with degree growth bounded by $f$.

\begin{theorem}\label{orderbound}
Let $F,K \subseteq \K\{Y\}$ be finite subsets with $h = \mathcal{W}(F \cup K)$, $L = \mathfrak{L}^n_{f,m}$, and $\mathcal{A} = \mbox{\sf Rosenfeld-Gr\"obner}(F,K)$, where $f(i) = hf_{i}$ with $\{f_i\}$ the Fibonacci sequence.  Then
\[
\mathcal{W}\left(\bigcup_{(A,H) \in \mathcal{A}} (A \cup H)\right) \leqslant hf_{L + 1}.
\]
\end{theorem}

\begin{proof}
Since $w(\mbox{pd-red}(p,B)) \leqslant w(p)$ for any $p \in \K\{Y\}$ and weak d-triangular set $B$, we have $\mathcal{W}(B \cup K) \leqslant \mathcal{W}(A \cup H)$, where $\{(B,K)\} = \mbox{\sf auto-partial-reduce}(A, H)$.  Hence, it suffices to bound $\mathcal{W}(G \cup D \cup A \cup H)$ whenever the tuple $(G,D,A,H)$ is added to $\mathcal{S}$ in {\sf Rosenfeld-Gr\"obner}.  

By Theorem~\ref{antichainbound} and the correspondence between antichain sequences of $\Theta Y$ and $\ZN^m \times \mathfrak{n}$, we obtain an antichain sequence of $\ZN^m \times \mathfrak{n}$ of degree growth bounded by $f(i)$, so the length of this sequence (and thus the sequence from Theorem~\ref{antichainbound}) is at most $L$.

In the proof of Theorem~\ref{antichainbound}, it is shown that for all $i\leqslant N$, for $j := \antik_i$, we have
\[
\mathcal{W}\left(\bigcup_{t = 1}^i (G_t \cup D_t \cup A_t \cup H_t)\right) \leqslant hf_{j + 1}.
\]
Since the largest possible $j$ is the length of the antichain sequence (and this $j$ is equal to $\antik_N$), for every $(G_i,D_i,A_i,H_i)$, we have 
\[
\mathcal{W}(G_i \cup D_i \cup A_i \cup H_i) \leqslant hf_{L + 1}.
\]
Since every $(G,D,A,H)$ added to $\mathcal{S}$ is equal to $(G_i,D_i,A_i,H_i)$ for some $i$, this ends the proof.
\end{proof}

\begin{corollary}\label{m2_orderbound}
	Let $m = 2$, $F,K \subseteq \K\{Y\}$ be finite subsets with $h = \mathcal{W}(F \cup K)$, $L = \mathfrak{L}^n_{f,m}$, and $\mathcal{A} = \mbox{\sf Rosenfeld-Gr\"obner}(F,K)$,  where $f(i) = f(n, h)_{i}$ with $\{f(n, h)_i\}$  given by~\eqref{eq:m2_bound}.
    Then
\[
	\mathcal{W}\left(\bigcup_{(A,H) \in \mathcal{A}} (A \cup H)\right) \leqslant f(n, h)_{L + 1}.
	\]	
\end{corollary}

\begin{proof}
	Replacing $hf_i$ with $f(n, h)_i$ everywhere in the proof of Theorem~\ref{orderbound}, we obtain an argument that is valid in all cases except for the case in which, for every $q$, $1 \leqslant q \leqslant n$, the antichain sequence $\{ s_j \}$ contains $\partial_1^{a_q}y_q$ and $\partial_2^{b_q}y_q$ for some $a_q$ and $b_q$.
    In this case, we still have  $\mathcal{W}(A_i) \leqslant f(n, h)_L$ for all $i$.
    We will prove that $\mathcal{W}(H_i) \leqslant f(n, h)_{L + 1}$ for all $i$.
    For $i < k_L$, this inequality follows from the proof of Theorem~\ref{antichainbound}.
    For $i \geqslant k_L$, every $h$ added to $H_i$ is reduced with respect to $A_i$ (see {\sf Rosenfeld-Gr\"obner}).
    The definition of $k_j$ implies that the set of leaders of $A_{k_j}$ contains $s_j$.
    While performing {\sf update} for $A_i$, every leader $s$ of $A_i$ either survives or is replaced with $\tilde{s}$ such that $s$ is a derivative of $\tilde{s}$.
    Hence, for all $i \geqslant k_j$, the set of leaders of $A_i$ contains either $s_j$ or $\tilde{s}$ such that $s_j$ is a derivative of $\tilde{s}$.
    Thus, since $h$ is reduced with respect to $A_i$ for $i \geqslant k_L$, for every variable $\partial_{1}^a \partial_2^b y_q$ occurring in $h$, we have $a < a_q$ and $b < b_q$.
    Thus, 
    \[
    w(h)\leqslant\max\limits_{1 \leqslant q \leqslant n} \left( w\left(\partial_1^{a_q - 1}y_q\right) + w\left( \partial_2^{b_q - 1}y_q \right) \right) \leqslant f(n, h)_{L} + f(n, h)_{L - 1} - 2 < f(n, h)_{L + 1}.\qedhere
    \]
\end{proof}

We can use Theorem~\ref{orderbound} and Corollary~\ref{m2_orderbound} to bound the orders of the output {\sf Rosenfeld-Gr\"obner}.  Let $F,K \subseteq \K\{Y\}$ be two finite subsets, and define a weight $w$ on $\Theta$ such that 
\begin{equation}\label{weightorder}
\mathcal{W}(F \cup K) = \mathcal{H}(F \cup K).
\end{equation}
This can always be done by letting $w(\theta) = \ord(\theta)$ for all derivatives $\theta$, but there are sometimes other weights that lead to equation~\eqref{weightorder} being satisfied.

\begin{example} We provide examples of differential polynomials $f$ that arise as part of systems of PDEs for which it is possible to construct a nontrivial weight $w$ such that $w(f) = \ord(f)$.  We note that we are not applying {\sf Rosenfeld-Gr\"obner} to these examples; we simply present them to demonstrate that there are nontrivial weights satisfying equation~\eqref{weightorder}.
\begin{enumerate}
\item Consider the heat equation
\[
u_t - \alpha\cdot (u_{xx} + u_{yy}) = 0,\quad f(u) := \partial_tu - \alpha\cdot(\partial_x^2u + \partial_y^2u) \in \K\{u\},
\]
where $u(x,y,t)$ is the unknown,   $\alpha$ is a positive constant, and $\K\{u\}$
has derivations $\{\partial_x,\partial_y,\partial_t\}$.  If we define a weight $w$ on $\Theta$ by 
\[
w\left(\partial_x^i\partial_y^j\partial_t^k\right) = i+j+2k,
\]
then $w(f) = 2 = \ord(f)$.

\item Consider the K-dV equation
\[
\phi_t + \phi_{xxx} + 6\phi\phi_x = 0,\quad f(\phi) := \partial_t\phi + \partial_x^3\phi + 6\phi\partial_x\phi \in \K\{\phi\},
\]
where $\phi(x,t)$ is the unknown and $\K\{\phi\}$ has derivations $ \{\partial_x,\partial_t\}$. Define a weight $w$ on $\Theta$ by
\[
w\left(\partial_x^i\partial_t^j\right) = i + 3j,
\]
so that $w(f) = 3 = \ord(f)$.
\end{enumerate}
\end{example}

Using Theorem~\ref{orderbound}, Corollary~\ref{m2_orderbound}, and~\eqref{weightorder}, we obtain the following order bound for the output of {\sf Rosenfeld-Gr\"obner}:

\begin{corollary}\label{refinedbound}
Let $F,K \subseteq \K\{Y\}$ be finite subsets with $h = \mathcal{H}(F \cup K)$, $L = \mathfrak{L}^n_{f,m}$, $\mathcal{A} = \mbox{\sf Rosenfeld-Gr\"obner}(F,K)$, where $f(i) = f(n,h)_i$ with $\{f(n,h)_i\}$ the sequence given by~\eqref{eq:m2_bound} if $m=2$ and $f(i) = hf_i$ with $\{f_i\}$ the Fibonacci sequence if $m > 2$.  Let $w\left(\partial_1^{i_1}\cdots\partial_m^{i_m}\right) = c_1i_1 + \cdots + c_mi_m$ be a weight defined on $\Theta$ such that $\mathcal{W}(F \cup K) = \mathcal{H}(F \cup K)$.  Then, for all $g \in \mathcal{A}$,
\[
\ord(g,\partial_i) \leqslant \begin{cases} \frac{f(n,h)_{L+1}}{c_i} &\text{if }m=2 \\ 
\frac{hf_{L+1}}{c_i} &\text{if } m > 2. \end{cases}
\]
\end{corollary}

%%%%%%%%%%%%%%%%%%%%%%%%%%%%%%%%%%%%%%%%%%%%%%%%%%%%%%%%%%%%%%%%%%%%%%%%%%%%%%%%%%%%%%%%%%%%%%%%%%%%%%%%%

\section{Specific values}\label{sec:specific}
In order to apply the results of the previous section, we need to be able to effectively compute $\mathfrak{L}^n_{f,m}$.  \citep{Pierce} only proved the existence of this number, without an analysis of how to construct it.  \citep{FLS} constructed an upper bound for $m=1,2$.  The first analysis for the case of arbitrary $m$ appears in \citep{LSO}.  

Let $f\colon \Z_{>0} \to \ZN$ be an increasing function.
Let us define a function $\Psi_{f, m} \colon \Z_{>0} \times \ZN^m \to \ZN$ by the following relations:
\[\begin{cases}
\Psi_{f, m}(i, (0, \ldots, 0, u_m)) = i,
\\
\Psi_{f, m}(i - 1, (u_1, \ldots, u_r, 0, \ldots, 0, u_m))\\\quad\quad= \Psi_{f, m}(i, (u_1, \ldots, u_r - 1, f(i) - f(i - 1) + u_m + 1, 0, \ldots, 0))
, & r < m - 1, u_r > 0,\\
\Psi_{f, m}(i - 1, (u_1, \ldots, u_m))\\   \quad\quad = \Psi_{f, m}(i, (u_1, \ldots, u_{m - 1} - 1, f(i) - f(i - 1) + u_m + 1)),
 &u_{m - 1} > 0.
\end{cases}\]
\begin{proposition}[{\citep[Corollary~3.10]{LSO}}]
The maximal length of an antichain sequence in $\ZN^m$ with degree growth bounded by $f$ does not exceed 
\[
\Psi_{f, m}(1, (f(1), 0,\ldots, 0)).
\]
\end{proposition}

Let us also define the sequence $\psi_0, \psi_1, \ldots$ by the relations $\psi_0 = 0$ and \[\psi_{i + 1} = \Psi_{f_i, m} (1, (f_i(1), 0, \ldots, 0)) + \psi_i,\quad f_i(x) := f(x + \psi_i).\]

\begin{proposition}[{\citep[Corollary~3.14]{LSO}}]\label{antichainlength} The maximal length of an antichain sequence in $\ZN^m \times \mathfrak{n}$ with degree growth bounded by $f$ does not exceed $\psi_n$.
\end{proposition}

Now, let us apply this technique to the functions $f_1(i) = f(n,h)_i$ and $f_2(i) = h f_i$.  Then, by Theorem~\ref{orderbound} and Corollary~\ref{m2_orderbound}, an upper bound on the weights of the output of  {\sf Rosenfeld-Gr\"obner} will be $f_1(\mathfrak{L}^n_{f_1,m}+1)$ if $m=2$ and $f_2(\mathfrak{L}^n_{f_2,m}+1)$ if $m > 2$.  In general, we do not have a formula for $\mathfrak{L}^n_{f_1,m}$ and $\mathfrak{L}^n_{f_2,m}$ for arbitrary $h,m,n$ that improves the one given in Proposition~\ref{antichainlength}; however, we can compute $\mathfrak{L}^n_{f_1,m}$ and $\mathfrak{L}^n_{f_2,m}$ for some specific values of $h,m,n$.

If $\mathcal{W}(F \cup K) = \mathcal{H}(F \cup K) = h$, we can use Corollary~\ref{refinedbound} to produce perhaps sharper bounds for the order of the elements of $\mbox{\sf Rosenfeld-Gr\"obner}(F,K)$ with respect to particular derivations.  In the examples that follow, we calculate upper bounds for $\ord(g,\partial_1)$ for $g \in \mbox{\sf Rosenfeld-Gr\"obner}(F,K)$, where $w\left(\partial_1^{i_1}\cdots\partial_m^{i_m}\right) = c_1i_1+\cdots+c_mi_m$ in the case in which $c_1=2$ and the case in which $c_1=3$.  We note that in the tables that follow, ``N/A" appears whenever we cannot have the given initial order $h$ with given $c_i$ as part of the weight function.

\begin{enumerate}
	\item Assume that $n = 1$ and $m = 2$. Then the maximal length of an antichain sequence does not exceed $h + 1$.
    In this case, the weights of the resulting polynomials are bounded by $f(1,h)_{h+2}$, which results in the following table:
  \begin{center}
    \begin{tabular}{ | c | c| c | c |c |c |c | c|c|c|c|}
    \hline
    $h$ & 1 & 2& 3& 4 &5 &6& 7&8&9&10\\  \hline
    $f(1,h)_{h+2}$& 1 & 4 & 11 & 25 & 55 & 106 & 205 & 386 & 713 & 1297 \\\hline
    $\ord(g,\partial_1)$, $c_1=2$ & N/A & 2 & 5 & 12 & 26 & 53 & 102 & 193 & 356 & 648 \\ \hline
    $\ord(g,\partial_1)$, $c_1 = 3$ & N/A & N/A & 3 & 8 & 17 & 35 & 68 & 128 & 237 & 432 \\ \hline 
    \end{tabular}
\end{center}

   \item Assume that $m = 2$ and $n$ is arbitrary.
    Then the maximal length of an antichain sequence does not exceed $b_n$, where $b_n$ satisfies $b_1 = h + 1$ and $b_{n + 1} = f(n,h)_{b_n + 1} + b_n + 1$, which results in the following table:

\begin{center}
    \begin{tabular}{ | c | c| c|c| c | c |}
    \hline
    $n$ & $h$& $b_n$& $f(n,h)_{b_n+1}$ & $\ord(g,\partial_1)$, $c_1 = 2$ & $\ord(g,\partial_1)$, $c_1 = 3$\\ \hline
    2&1 & 5 &4 & N/A & N/A\\ \hline
    2&2  & 9 &77 & 38 & N/A \\ \hline
    2 &3 &18 &9,960 & 4,980 & 3,320 \\
    \hline
       2 &4 &34 &$31,206,974$ & $15,603,487$ & $10,402,324$ \\
    \hline
       3 &1 &11 &90 & N/A & N/A \\
    \hline
    \end{tabular}
\end{center}  

\item Assume that $m=3$ and $n=1$.  We can construct the maximal length antichain sequence of $\ZN^3$ using the methods of \citep{LSO} and the function $f(i) = hf_i$, resulting in the following sequence:
\begin{multline*}
(h,0,0), (h-1,1,0),(h-1,0,h+1),(h-2,2h+2,0),\dots, \\
(h-2,0,hf_{2h+6}-(h-2)),\dots,(h-i,hf_{c_{i-1}+1}-(h-i),0), \dots, \\
(h-i,0,hf_{c_i}-(h-i)),\dots, (0,hf_{c_{h-1}+1},0),\dots,(0,0,hf_{c_h}),
\end{multline*}
where the sequence $c_i$ is given by $c_0 = 1$ and for $1 \leqslant i \leqslant h$, 
\[
c_i = c_{i-1} + 1 + hf_{c_{i-1}+1}-(h-i).
\]
As a result, we see that the maximal length of an antichain sequence is equal to $c_h$.
\end{enumerate}

Below is a table of some maximal lengths $\mathfrak{L}^n_{f,m}$ and weights $f(\mathfrak{L}^n_{f,m}+1)$, where $f(i) = hf_i$, for $m=3$, $4$, and $5$:

    \begin{center}
    \begin{tabular}{ |c | c | c| c|c|}
    \hline
    $m$& $n$ & $h$& length & weight \\ \hline
 3&    1&1 & 3 &3 \\ \hline
 3&   1&2  & 10 & 178 \\ \hline
    4&      1 &1 & 5 & 8 \\
    \hline
    5&      1 &1 & 20 & 10,946 \\
    \hline
    \end{tabular}
\end{center}

\section{Order lower bound}\label{sec:lower}
This section gives a lower bound for the orders of the output of {\sf Rosenfeld-Gr\"obner}, coming from the lower bound for degrees of elements of a Gr\"obner basis from \citep{Yap}. To be specific, we show that for $m,h$ sufficiently large, there is a set of $r$ differential polynomials $F \subseteq \K\{y\}$ of order at most $h$, where $\K$ is equipped with $m$ derivations, $r \sim m/2$, and $\K$ is constant with respect to all of the derivations, such that if $\mathcal{A} = \mbox{\sf Rosenfeld-Gr\"obner}(F,\{1\})$, then
\begin{equation}\label{eq:lowerbound}
\mathcal{H}\left(\bigcup_{(A,H) \in \mathcal{A}} (A \cup H)\right) \geqslant h^{2^r}.
\end{equation}
The arguments presented here are standard, and we include them for completeness.  We first note the following standard fact about differential ideals generated by linear differential polynomials.

\begin{proposition}\label{lineardecomp}
Suppose $F,K \subseteq \K\{Y\}$ are composed of linear differential polynomials.  Then the output of $\mbox{\sf Rosenfeld-Gr\"obner}(F,K)$ is either empty or consists of a single regular differential system $(A,H)$ with $A$ and $H$ both composed of linear differential polynomials.
\end{proposition}

Suppose now we apply {\sf Rosenfeld-Gr\"obner} to $(F,\{1\})$, where $F$ consists of linear differential polynomials, in order to obtain a regular decomposition of $\{F\}$.  Since every element of $F$ is linear, $[F]$ is a prime differential ideal, so by Proposition~\ref{lineardecomp}, we have
\[
[F] = \{F\} = [A]:H^\infty
\]
for some regular differential system $(A,H)$, with $A$ and $H$ both composed of linear differential polynomials.  Since every element of $A$ is linear, after performing scalar multiplications and addition, $A$ can be transformed to an autoreduced set $\bar{A}$ without affecting the leaders and orders of elements of $A$.  Since $(A,H)$ is a regular differential system, $\bar{A}$ is a characteristic set of $[F]$.  So, it suffices to find a lower bound on the orders of elements of linear characteristic sets in $\K\{Y\}$.

There is a well-studied one-to-one correspondence between polynomials in $\K[x_1,\dots,x_m]$ and homogeneous linear differential polynomials in $\K\{y\}$ with $m$ derivations and $\K$ a field of constants:
\begin{equation}\label{correspondence}
\sum c_{i_1,\dots,i_m}x_1^{i_1}\cdots x_m^{i_m} \leftrightarrow \sum c_{i_1,\dots,i_m}\partial_1^{i_1}\cdots \partial_m^{i_m}y.
\end{equation}
Any orderly ranking on $\Theta y$ then determines a graded monomial order on $\K[x_1,\dots,x_m]$.

Given a polynomial $f \in \K[x_1,\dots,x_m]$, let $\tilde{f} \in \K\{y\}$ be its corresponding differential polynomial under \eqref{correspondence}.  By the discussion above, if we have a collection of polynomials $f_1,\dots,f_r \in \K[x_1,\dots,x_m]$, we can construct a characteristic set $C = \{C_1,\dots,C_s\}$ of $[\tilde{f}_1,\dots,\tilde{f}_r] \subseteq \K\{y\}$ consisting of homogeneous linear differential polynomials, and so each $C_i \in \K\{y\}$ is in fact equal to $\tilde{g}_i$ for some $g_i \in \K[x_1,\dots,x_m]$.  

\begin{proposition}[{cf.  \citep[page~352]{Wu},\citep{Gerdt}}]\label{polynomialgrobner}
With the notation above, $\{g_1,\dots,g_s\} \subseteq \K[x_1,\dots,x_m]$ is a Gr\"obner basis of the ideal $I = (f_1,\dots,f_r)$.
\end{proposition}

By Proposition~\ref{polynomialgrobner}, we can thus find a lower bound for the orders of the output of {\sf Rosenfeld-Gr\"obner} via a lower bound for the degrees of elements of a Gr\"obner basis, as we do in the following example.

\begin{example}
This example demonstrates the lower bound \eqref{eq:lowerbound} for the orders of the output of {\sf Rosenfeld-Gr\"obner}.  In \citep[Section~8]{Yap}, for $m,h$ sufficiently large, a collection of $m$ algebraic polynomials $f_1,\dots,f_r$ of degree at most $h$ in $m$ algebraic indeterminates, with $r \sim m/2$, is constructed such that any Gr\"obner basis of $(f_1,\dots,f_r)$ with respect to a graded monomial order has an element of degree at least $h^{2^r}$.  

As a result of the previous discussion, we have a collection of differential polynomials $F = \tilde{f}_1,\dots,\tilde{f}_r \in \K\{y\}$ of order $h$ with $m$ derivations such that any linear characteristic set of $[\tilde{f}_1,\dots,\tilde{f}_r]$ will contain a differential polynomial of order at least $h^{2^r}$.  Since in this case $\{(A,H)\} = \mbox{\sf Rosenfeld-Gr\"obner}(F,\{1\})$ can be transformed into a linear characteristic set without affecting the orders of the elements, this means that
\[
\mathcal{H}(A \cup H) \geqslant h^{2^r}.
\]
\end{example}

\bibliographystyle{elsarticle-harv}
\bibliography{bibdata}
\end{document}